\documentclass{article}
\usepackage{geometry}
\usepackage{graphicx,tikz,caption,subcaption}
\usepackage{fullpage}
\usepackage{hyperref}
\usepackage{enumitem,verbatim}
\usepackage{hyperref}
\usepackage{amsmath}
\usepackage{amssymb}
\usepackage{amsfonts}
\usepackage{amsthm}
\usepackage[capitalise]{cleveref}
\usepackage{geometry}
\usepackage{float}
\newtheorem{innerproblem}{Problem}
\newenvironment{problem*}
  {
   \innerproblem}
  {\endinnerproblem}
\newtheorem{innerdefinition}{Definition}
\newenvironment{definition*}
  {
   \innerdefinition}
  {\endinnerdefinition}
\newtheorem{theorem}{Theorem}[section]
\newtheorem{lemma}[theorem]{Lemma}
\newtheorem{definition}[theorem]{Definition}
\newtheorem{claim}[theorem]{Claim}
\newtheorem{cor}[theorem]{Corollary}

\newtheorem{remark}{Remark}

\newtheorem{observation*}{Observation}[section]

\newtheorem{problem}{Problem}[section]

\newcommand{\ex}{\mathrm{ex}}
\usetikzlibrary{patterns}
\usetikzlibrary{shapes}
\usetikzlibrary{decorations.pathreplacing}
\usetikzlibrary{decorations.pathmorphing}
\usetikzlibrary{positioning}
\definecolor{mypink}{RGB}{255, 105, 160}
\definecolor{myorange}{RGB}{255, 178, 49}
\definecolor{darktangerine}{rgb}{1.0, 0.66, 0.07}
\definecolor{darkpastelgreen}{rgb}{0.01, 0.75, 0.24}
\title{Counting cliques without generalized theta graphs}
\author{Jun Gao\thanks{Mathematics Institute and DIMAP,
            University of Warwick,
            Coventry, UK, and Extremal Combinatorics and Probability Group (ECOPRO), Institute for Basic Science (IBS),  Daejeon, South Korea. Supported by ERC Advanced Grant 101020255 and the Institute for Basic Science (IBS-R029-C4). E-mail: jungao@ibs.re.kr}
\and Zhuo Wu\thanks{Departament de Matemàtiques, Universitat Politècnica de Catalunya (UPC),
Carrer de Pau Gargallo 14, 08028 Barcelona, Spain. Z. Wu acknowledges the bilateral AEI+DFG research project PCI2024-155080-2: SRC-ExCo – Structure, Randomness and Computational Methods in Extremal Combinatorics, and the PID2023-147202NB-I00 (COCOA: COntemporary COmbinatorics and its Applications), 
all funded by MICIU/AEI/10.13039/501100011033. Email: \texttt{zhuo.wu@upc.edu}}
\and Yisai Xue\thanks{\textit{Corresponding author}. School of Mathematics and Statistics, Ningbo University, Ningbo, China and Extremal Combinatorics and Probability Group (ECOPRO),  Institute for Basic Science (IBS), Daejeon, South Korea. Supported by the Zhejiang Provincial Natural Science Foundation of China (No. LQN25A010024),  China Scholarship Council and the Institute for Basic Science (IBS-R029-C4). E-mail: xueyisai@nbu.edu.cn}}

\begin{document}

\maketitle

\begin{abstract}
 The \textit{generalized Tur\'an number} $\ex(n, T, F)$  is the maximum possible number of copies of $T$ in an $F$-free graph on $n$ vertices for any two graphs $T$ and $F$. 
 For the book graph $B_t$, there is a close connection between $\ex(n,K_3,B_t)$ and the Ruzsa-Szemer\'edi triangle removal lemma. 
 Motivated by this, in this paper, we study the generalized Tur\'an problem for generalized theta graphs, a natural extension of book graphs.
 Our main result provides a complete characterization of the magnitude of $\ex(n,K_3,H)$ when $H$ is a generalized theta graph, indicating when it is quadratic, when it is nearly quadratic, and when it is subquadratic. 
 Furthermore, as an application, we obtain the exact value of $\ex(n, K_r, kF)$, where $F$ is an edge-critical generalized theta graph, and $3\le r\le k+1$, extending several recent results.

% \bigskip

% \noindent{\bf Keywords:}  generalized Tur\'an number, theta graph
% \medskip

% \noindent{\bf AMS (2000) subject classification:}  05C35
 \end{abstract}

\section{Introduction}

  The \textit{Tur\'an number} of a graph $F$, denoted by $\mathrm{ex}(n, F)$, is the maximum possible number of edges in an $F$-free graph with $n$ vertices.
  The problem of determining Tur\'an number for assorted graphs traces its history back to 1907, when Mantel showed that $\mathrm{ex}(n, K_3)=\lfloor n^2/4\rfloor$. 
  In 1941, Tur\'an proved that if a graph does not contain a complete subgraph $K_{r+1}$, then the maximum number of edges it can contain is given by the \textit{Tur\'an graph} $T_r(n)$, i.e., the complete balanced $r$-partite graph on $n$ vertices. 

%  In order to get a better understanding of these problems, mathematicians generalized the definition of the Tur\'an number. 
  Given two graphs $T$ and $F$, the \textit{generalized Tur\'an number} of $F$ is the maximum possible number of copies of $T$ in an $F$-free graph on $n$ vertices, denoted by $\mathrm{ex}(n, T, F)$. 
  Note that $\mathrm{ex}(n, K_2, F)=\mathrm{ex}(n, F)$. 
  The systematic study of $\mathrm{ex}(n, T, F)$ was initiated by Alon and Shikheman \cite{alon2016many}. 
  Among others, they bounded the maximum number of triangles in $C_{2\ell+1}$-free graphs, showing that
\begin{align*}
  \mathrm{ex}(n,K_3,C_{2\ell+1})\leq \frac{16(\ell-1)}{3} \mathrm{ex}(\lceil n / 2\rceil, C_{2\ell})=O(n^{1+1/\ell}).
\end{align*}
  For even cycles, F\"uredi and \"Ozkahya \cite{furedi20173} proved the following bound.
\begin{align*}
  \mathrm{ex}(n,K_3,C_{2\ell})\leq \frac{2(\ell-3)}{3} \mathrm{ex}(n, C_{2\ell})=O(n^{1+1/\ell}).
\end{align*}
  Since the work of Alon and Shikheman, there has been much research focusing on the generalized Tur\'an number of cycles, see e.g.  \cite{beke2023generalized}, \cite{gerbner2020generalized}, \cite{gishboliner2020generalized}.

  The \textit{book graph} $B_t$ is the graph consisting of $t\ge 2$ triangles, all sharing one common edge. 
  Alon and Shikheman \cite{alon2016many} also showed that 
\begin{align}\label{thm:book}
    n^{2-o(1)}\le \ex(n,K_3,B_t) = o(n^2). \tag{$\diamondsuit$}
\end{align}
 The upper bound is obtained by using the triangle removal lemma, and the lower bound is derived from the construction introduced by Ruzsa and Szemer\'edi \cite{ruzsa1978triple}, which is based on Behrend's method \cite{behrend1946sets} for creating dense subsets of the first $n$ integers without any three-term arithmetic progressions.
 It is conjectured \cite{1981On} that for any graph $F$, there exists a constant $\alpha$ such that $\ex(n,F)=\Theta(n^{\alpha})$, so (\ref{thm:book}) reveals the difference between Tur\'an number and generalized Tur\'an number.
% Note that for any graph $G$, there exists a constant $\alpha$ such that $\ex(n,G)=\Theta(n^{\alpha})$, so this phenomenon reveal the difference between Tur\'an number and generalized Tur\'an number.

In this paper, we will focus on the extremal function of the \textit{generalized theta graph}. More precisely, we make the following definition.

\begin{definition}
 Let $k\ge 2$ and $p_1$, $\ldots$, $p_k\ge 1$ be integers. 
 The \textit{generalized theta graph} $\Theta(p_1,\ldots,p_k)$ consists of a pair of end vertices  joined by $k$ internally disjoint paths of lengths $p_1$, $\ldots$, $p_k$, respectively.
\end{definition}

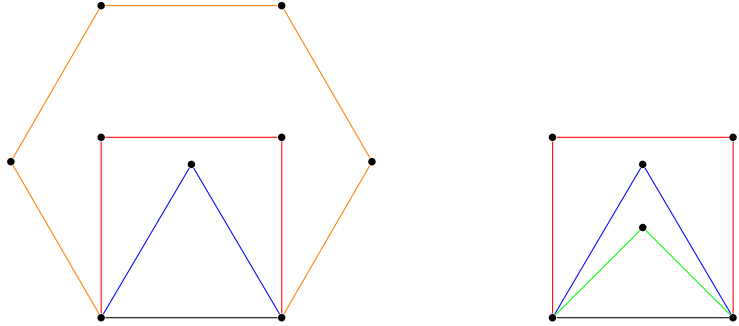
\begin{figure}[!ht]
    \centering   
\begin{tikzpicture}[scale=1.2]
\node[inner sep= 1pt](u) at (3,0)[circle,fill]{};
\node[inner sep= 1pt](v) at (5,0)[circle,fill]{};  
\draw[line width=0.901pt] (u) -- (v);

\node[inner sep= 1pt](w) at (4,1.7)[circle,fill]{};
\draw [blue,line width=0.901pt](u) -- (w);
\draw [blue,line width=0.901pt](w) -- (v);

\node[inner sep= 1pt](x1) at (3,2)[circle,fill]{};
\node[inner sep= 1pt](x2) at (5,2)[circle,fill]{};  
\draw [red,line width=0.901pt](u) -- (x1);
\draw [red,line width=0.901pt](x1) -- (x2);
\draw [red,line width=0.901pt](x2) -- (v);

%\node[inner sep= 1pt](y1) at (2.39, 1.9)[circle,fill]{};
%\node[inner sep= 1pt](y2) at (4, 3.08)[circle,fill]{}; 
%\node[inner sep= 1pt](y3) at (5.62, 1.9)[circle,fill]{}; 
%\draw [purple](u) -- (y1);
%\draw [purple](y1) -- (y2);
%\draw [purple](y2) -- (y3);
%\draw [purple](y3) -- (v);

\node[inner sep= 1pt](z1) at (2, 1.73)[circle,fill]{};
\node[inner sep= 1pt](z2) at (3, 3.46)[circle,fill]{}; 
\node[inner sep= 1pt](z3) at (5, 3.46)[circle,fill]{}; 
\node[inner sep= 1pt](z4) at (6, 1.73)[circle,fill]{}; 
\draw [orange,line width=0.901pt](u) -- (z1);
\draw [orange,line width=0.901pt](z1) -- (z2);
\draw [orange,line width=0.901pt](z2) -- (z3);
\draw [orange,line width=0.901pt](z3) -- (z4);
\draw [orange,line width=0.901pt](z4) -- (v);

%\node[inner sep= 1pt](a1) at (1.75, 1.56)[circle,fill]{};
%\node[inner sep= 1pt](a2) at (2.2, 3.51)[circle,fill]{}; 
%\node[inner sep= 1pt](a3) at (4, 4.38)[circle,fill]{}; 
%\node[inner sep= 1pt](a4) at (5.8, 3.51)[circle,fill]{}; 
%\node[inner sep= 1pt](a5) at (6.25, 1.56)[circle,fill]{}; 
%\draw [green](u) -- (a1);
%\draw [green](a1) -- (a2);
%\draw [green](a2) -- (a3);
%\draw [green](a3) -- (a4);
%\draw [green](a4) -- (a5);
%\draw [green](a5) -- (v);
\node[inner sep= 1pt](p) at (8,0)[circle,fill]{};
\node[inner sep= 1pt](q) at (10,0)[circle,fill]{};  
\draw[line width=0.901pt] (p) -- (q);
\node[inner sep= 1pt](p1) at (8,2)[circle,fill]{};
\node[inner sep= 1pt](p2) at (10,2)[circle,fill]{};
\node[inner sep= 1pt](p_3) at (9,1.7)[circle,fill]{};
\node[inner sep= 1pt](p_4) at (9,1)[circle,fill]{};
\draw [blue,line width=0.901pt](p) -- (p_3);
\draw [blue,line width=0.901pt](q) -- (p_3);
\draw [green,line width=0.901pt](p) -- (p_4);
\draw [green,line width=0.901pt](q) -- (p_4);
\draw [red,line width=0.901pt](p) -- (p1);
\draw [red,line width=0.901pt](p2) -- (p1);
\draw [red,line width=0.901pt](q) -- (p2);
\end{tikzpicture}
\caption{The generalized theta graphs $\Theta(1,2,3,5)$ and $\Theta(1,2,2,3)$.}\label{fig-theta}
\end{figure}

\vskip 0.5 em

 As the generalized theta graph is a common generalization of both cycles and book graphs, there has been much interest in studying the extremal problem for generalized theta graphs, e.g.  \cite{bukh2020turan}, \cite{conlon2019graphs}, \cite{faudree1983class}, \cite{liu2023turan}.
 
  Based on results for cycles and book graphs, a natural question is, what is the asymptotic order of $\ex(n, K_3, H)$ when $H$ is a generalized theta graph.
 Here is our first main result.

\begin{theorem}\label{thm:jun2}
    Let $H = \Theta(p_1,p_2,\cdots, p_k)$ be a generalized theta graph.
    \begin{enumerate}
        \item[\rm 1.] If $H$ contains at most one triangle, then there exists some constant $\alpha$ such that $\ex(n,K_3,H) = O(n^{2-\alpha}).$ 
        \item[\rm 2.] If $H$ contains at least two triangles and does not contain $\Theta(1,2,2,3)$ as a subgraph, then $n^{2-o(1)}\le \ex(n,K_3,H) = o(n^2)$.
        \item[\rm 3.] Otherwise we have $\ex(n,K_3,H) = \Theta(n^{2}).$
    \end{enumerate}
\end{theorem}

\begin{remark}
   \cref{thm:jun2} provides a complete characterization of the magnitude of $\ex(n,K_3,H)$.
  That is, we determine when $\ex(n,K_3,H)$ is subquadratic, when it is nearly quadratic, and when it is quadratic.
\end{remark}

\begin{remark}
   Combining the results of \cref{lem:k-tree}(2) proved in \cref{sec3}, we extend the results for $K_3$ to $K_r$ by showing that $\ex(n,K_r,H) = O(\ex(n,K_3,H))$, where $H$ is a generalized theta graph.
\end{remark}

 Denote $kF$ as the union of $k$ disjoint copies of $F$. In \cite{alon2016many}, Alon and Shikhelman observed an interesting phenomenon that while $\ex(n,F)$ and $\ex(n, kF)$ are close to each other, the number of copies of other graphs can be far from each other if we forbid $F$ or $kF$. 
 This phenomenon was further examined by Gerbner, Methuku and Vizer \cite{gerbner2019generalized}, and they proved that $\ex(n,K_r,kF)=O(\max\limits_{m\le r}\ex(n,K_m,F))$.
 Furthermore, they showed that

 \begin{itemize}
  \item  If $r \leq k$, then $\ex(n, K_r, k C_{2\ell+1})=\Theta(n^2)$.
  \item If $r>k+1$, then $\ex(n, K_r, k C_{2\ell+1})=O(n^{1+1 / \ell})$.
\end{itemize}

 In addition to determining the asymptotic order of $\mathrm{ex}(n, K_r, k C_{2\ell+1})$, Gerbner, Methuku and Vizer \cite{gerbner2019generalized} also suggested determining its exact value.
 The first progress in this direction is that Zhang, Chen, Gy\H{o}ri and Zhu \cite{zhang2023maximum} obtained the exact value for $\mathrm{ex}(n, K_3, 2 C_{5})$.
 Subsequently, Hou, Yang and Zeng \cite{hou2023counting} provided the exact value for $\mathrm{ex}(n, K_3, k C_{2\ell+1})$.

  Our second result shall extend the above results. 
  A graph is \textit{edge-critical} if there exists an edge whose deletion reduces its chromatic number.
 Simonovits \cite{simonovits1968method} showed that  $T_r(n)$  is the unique graph which attains the maximum number of edges in an $n$-vertex $F$-free graph when $F$ is an edge-critical graph with $\chi(F)=r+1>2$.

 It is easy to observe that a generalized theta graph is edge-critical if and only if the length of all of its paths have the same parity except exactly one of them, and all the edge-critical generalized theta graphs have chromatic number 3.
 
  Here is our second main result.

\begin{theorem}\label{thm:main}
  Let $k\geq 2$ and $n$ be sufficiently large.
  Assume that $F$ is an edge-critical generalized theta graph. 
\begin{enumerate}
   \item[\rm 1.] For $3\leq r\le k+1$,
\begin{align*}
  \ex(n, K_r, kF)
  =\binom{k-1}{r} + \binom{k-1}{r-1}(n-k+1) + \binom{k-1}{r-2}\lfloor(n-k+1)^2/4\rfloor.
\end{align*}
% and $K_{k-1}\vee T_2(n-k+1)$ is the unique extremal graph.
   \item[\rm 2.] For $r\ge k+2$, $$\ex(n,K_r,kF)=o(n^2).$$
\end{enumerate}
\end{theorem}

Note that both odd cycles and book graphs are edge-critical generalized theta graphs.
\cref{thm:main} provides the exact values of $\ex(n, K_r, k C_{2\ell+1})$ and $\ex(n, K_r, k B_t)$ for $3\le r\le k+1$.

\section{Definitions and some auxiliary results}

\textbf{Notation.} All the graphs in this paper are simple graphs. Let $G$ be a graph. We will denote the set of vertices of $G$ by $V(G)$ and the set of edges by $E(G)$, and define $|G| := |V(G)|$ and $e(G):=|E(G)|$. For any $v\in V(G)$, let $N_G(v)$ be the set of the neighbours of $v$ in $G$, $d_G(v) :=|N_G(v)|$, and $\delta(G) :=\min\limits_{v\in V(G)} d_G(v)$. For any two graphs $G_1$ and $G_2$, let $G_1\vee G_2$ denote the graph obtained from $G_1 \cup G_2$ by adding all edges between $V(G_1)$ and $V(G_2)$. 
  Let $U\subseteq V(G)$ and define $G-U$ be $G[V(G)\backslash U]$.

 We use $[n]$ to  denote the set $\{1,2,\cdots,n\}$, and $K_n$, $C_n$, $P_n$ to denote the complete graph (clique), cycle, and path on $n$ vertices respectively.
 For any graph $G$, let $\mathcal{K}_{t}(G)$ be the collection of the $t$-cliques in $G$.
 For a generalized theta graph $\Theta(p_1,\ldots,p_k)$, we refer to the common end vertices of the $k$ paths as the \textit{root vertices}.
 We call the degree-2 vertices of a book graph the \textit{page vertices}.
 A \textit{face} is a connected region of the plane bounded by edges of the planar embedding of the graph.
 The \textit{boundary} of a face is the closed walk in $G$ bounding it, which is a cycle if the graph is simple and 2-connected.

  Let $f,g$ be two functions on $\mathbb{N}$. Define $f(n)=O(g(n))$ if there exists a constant $c$ such that for sufficiently large $n$, $f(n)<cg(n)$, and define $f(n)=o(g(n))$ if for any $\varepsilon>0$, $f(n)<\varepsilon g(n)$ for sufficiently large $n$.

\vspace*{0.2cm}
  
A \textit{$k$-tree} is a generalization of a tree that has the following recursive construction.
\begin{definition}[$k$-tree, Beineke and Pippert \cite{beineke1969number}]\label{ktree}
  Let $k$ be a fixed positive integer.
  \begin{enumerate}
    \item  The complete graph $K_k$ is a $k$-tree.
    \item If $T$ is a $k$-tree, then so is the graph obtained from $T$ by joining a new vertex to all vertices of some $k$-clique of $T$.
  \end{enumerate}
\end{definition}

Here are some auxiliary results that are needed in this paper.

\begin{theorem}[K{\H{o}}v{\'a}ri, S{\'o}s and Tur{\'a}n \cite{kHovari1954problem}]\label{thm:KST}
  For $2\le a\le b$,
\begin{align*}
   \mathrm{ex}(n, K_{a,b})=O(n^{2-1/a}).
\end{align*}
\end{theorem}

\begin{theorem}[Simonovits \cite{simonovits1968method}]\label{thm:simonovits}
  Let $F$ be an edge-critical graph with $\chi(F)=r+1>2$ and $n$ be sufficiently large. 
  Then the Tur\'an graph $T_r(n)$  is the unique graph which attains the maximum number of edges in an $n$-vertex $F$-free graph.
\end{theorem}

\begin{theorem}[Erd\H{o}s-Simonovits \cite{erdos1966limit}, F\"uredi \cite{furedi2015proof}]\label{Erdos-Simon}
  For any $\varepsilon>0$ and a graph $H$ with $\chi(H)=r+1$, there exists $n_0=n_0(\varepsilon, H)$ such that the following holds for all $n \geq n_0$. Let $G$ be an $n$-vertex $H$-free graph. If
\begin{align*}
  e(G) \geq\left(1-\frac{1}{r}-\varepsilon\right)\binom{n}{2}
\end{align*}
then $G$ can be obtained from $T_r(n)$ by adding and deleting a set of at most $8 \sqrt{\varepsilon} n^2$ edges.
\end{theorem}

%\textcolor{red}{\begin{theorem}[Nikiforov \cite{nikiforov2008graphs}]\label{thm:nikiforov}
%  Let $r \geq 2, c^r \ln n \geq 1$, and $G$ be a graph with $n$ vertices. Every set of at least $cn^r$ $r$-cliques of $G$ covers a $K_r(s, \ldots s, t)$ with $s=\lfloor c^r \ln n\rfloor$ and $t>n^{1-c^{r-1}}$.
%\end{theorem}}

\section{Forbidding generalized theta graphs}\label{sec3}
In this section we will prove \cref{thm:jun2}.  First we prove the following two lemmas.

\begin{lemma}\label{lem:k-tree}
    Let $H$ be a subgraph of a $k$-tree and $G$ be an $H$-free graph.
\begin{enumerate}    
    \item[\rm (1)] There exist a constant $C:=C(H)$ and a map $f:\mathcal{K}_{k+1}(G)\rightarrow\mathcal{K}_{k}(G)$ such that

    \begin{itemize}
    \item For each $U\in \mathcal{K}_{k+1}(G)$, $f(U)$ is a $k$-clique in $U$;
    \item For each $W\in \mathcal{K}_{k}(G)$, $r(W):=|f^{-1}(W)|\le C$.
    \end{itemize}

     \item[\rm (2)]  $\ex(n,K_{k+1},H) = O (\ex(n,K_{k},H)).$
\end{enumerate}       
\end{lemma}
\begin{proof}
    (1) We only need to show the case that $H$ is a $k$-tree.
    Choose a map $f:\mathcal{K}_{k+1}(G)\rightarrow\mathcal{K}_{k}(G)$ that satisfies the first condition and minimizes the function 
   \[\psi(f)=\sum_{W\in\mathcal{K}_{k}(G)}r(W)^2.\]

    \begin{claim}\label{cl:-1} 
     Let $U_0 \in \mathcal{K}_{k+1}(G)$ and define $f(U_0) = W_0$. For any $k$-clique $W'$ in $U_0$, $r(W')\ge r(W_0)-1$.
    \end{claim}

    \begin{proof}
    Let $g:\mathcal{K}_{k+1}(G)\rightarrow\mathcal{K}_{k}(G)$ be a map with
    \[g(U)=\begin{cases} f(U), \quad &\text{if}~ U\neq U_0;\\
    W', \quad &\text{if} ~U=U_0.
    \end{cases}\]
    By the minimality of $\psi(f)$, we have
    \[0\le \psi(g)-\psi(f)=(r(W_0)-1)^2+(r(W')+1)^2-r(W_0)^2-r(W')^2=2(r(W')-r(W_0)+1),\]
    which means $r(W')\ge r(W_0)-1$.
    \end{proof}

    By the definition of $k$-tree, we can obtain a sequence of $k$-trees $(H_1,H_2,\cdots ,H_m)$ such that $H_1$ is a $k$-clique, $H_m=H$, and $H_{i+1}$ is obtained from $H_i$ by joining a new vertex to all vertices of a $k$-clique $W_i$ in $H_i$.
    If there exists a $k$-clique $W_1'$ in $G$ such that $r(W_1') \ge C\ge 2|H|$,  we inductively show that for $1\le i \le m$, we can find a $k$-tree $Q_i$ of $G$ such that 
    \begin{itemize}
    \item $Q_i$ is isomorphic to $H_i$.
    \item $r(W)\ge 2|H|-i$ for any $k$-clique $W$ in $Q_i$.
    \end{itemize} 
    which contradicts the $H$-free property of $G$ since $H_m=H$.
    
    Let $Q_1=W_1'$. Assume that we have picked $Q_j$. 
    For convenience, we denote $V(f^{-1}(W))$ as the set of vertices of all $(k+1)$-cliques in $f^{-1}(W)$.
    For any $k$-clique $W$, let
    \[R(W)=\{u\in V(G):~u\in V(f^{-1}(W))\backslash V(W)\}.\]
    
    By the inductive hypothesis, we have $r(W)\ge 2|H|-j>|H|\ge |Q_j|$ for every $k$-clique $W$ in $Q_j$, then $R(W)\setminus V(Q_j)\neq \varnothing$. 
    As $Q_j$ is isomorphic to $H_j$, there is a $k$-clique $W_j'$ in $Q_j$, and a vertex $v\in R(W_j')\setminus V(Q_j)$ such that by joining vertex $v$ to all vertices of $W_j'$ of $Q_j$, the resulting graph, denoted as $Q_{j+1}$, is isomorphic to $H_{j+1}$.
    
    Now we check that $Q_{j+1}$ satisfies the second condition. 
    Let $U_j$ be the $(k+1)$-clique obtained by joining  $v$ to all vertices of $W_j'$, then $f(U_j)=W_j'$ since $v\in R(W_j')$. 
    For any $k$-clique $W\in\mathcal{K}_{k}(G)$ of $Q_{j+1}$, it is either a $k$-clique in $Q_j$, or a $k$-clique in $U_j$. For the first case, $|f^{-1}(W)|\ge 2|H|-j$ by induction hypothesis; For the second case, by \cref{cl:-1}, we have
    \[r(W)\ge r(W_j')-1\ge 2|H|-j-1,\]
    where the last inequality is concluded from $W_j'\subseteq Q_j$ and induction hypothesis. This finishes the proof of (1).  
    
    (2) Fix an $H$-free graph $G$ with $n$ vertices which maximizes the number of $(k+1)$-cliques, and choose a function $f$ and a constant $C$ that satisfy \cref{lem:k-tree} (1). Hence,  
    \[\ex(n,K_{k+1},H)=|\mathcal{K}_{k+1}(G)|\le C\cdot|\mathcal{K}_{k}(G)|\le C\cdot\ex(n,K_{k},H).\]
    
\end{proof}

It is easy to see that a generalized theta graph is a subgraph of some $2$-tree. Hence, we have the following corollaries of \cref{lem:k-tree}.

\begin{cor}\label{cor:theta}
  For any generalized theta graph $H$ and integer $r$, there is a constant $C$ such that 
  \begin{align*}
  \mathrm{ex}(n, K_r, H)\le Cn^2.
\end{align*}
\end{cor}

Note that a $1$-tree is just a tree. Hence, the following corollary is also obvious.
\begin{cor}\label{cor:tree}
For any tree $T$ and integer $r$, there is a constant $C$ such that 
\begin{align*}
\mathrm{ex}(n, K_r, T)\le Cn.
\end{align*}\end{cor}

The following lemma strengthens \cref{lem:k-tree} for generalized theta graphs.

\begin{lemma}\label{lem:theta graph}
    Let $H = \Theta(p_1,\cdots,p_k)$ be a generalized theta graph and $G$ be an $H$-free graph.  
    There exist a constant $C=C(H)$ and a map $f:\mathcal{K}_{3}(G)\rightarrow\binom{E(G)}{2}$ such that
    \begin{itemize}
    \item For each $T\in \mathcal{K}_{3}(G)$, $f(T)$ is a set of two different edges in $T$ ;
    \item For each $e\in E(G)$, $r(e)\le C$, where $r(e)$ is the number of the triangles $T$ such that $e\in f(T)$.
    \end{itemize}
\end{lemma}
\begin{proof}
    Let $C = 2|H|$.
    Choose a map $f:\mathcal{K}_{3}(G)\rightarrow\binom{E(G)}{2}$ that satisfies the first condition and minimizes the function 
   \[\Psi(f)=\sum_{e\in E(G)}r(e)^2.\]   
   
    For any triangle $T_i\in \mathcal{K}_{3}(G)$, if $f(T_i)=\{e_i,e_i'\}$, we call $e_i$ and $e_i'$ the \textit{corresponding edges} of $T_i$ (with respect to $f$),
    and define $e^*_i = E(T_i)\setminus \{e_i,e'_i\}$ as the \textit{non-corresponding edge}.
   The proof of the following claim is the same as that of \cref{cl:-1}.
    
\begin{claim}\label{cl:-12}    
    For any triangle $T_i$, $i\in[m]$, we have $r(e^*_i) \ge  \max\{r(e_i), r(e'_i)\}-1$.
\end{claim}

    To prove ~\cref{lem:theta graph}, we only need to show $r(e)\le C$ for each $e\in E(G)$.
    Suppose, to the contrary, that there exists an edge $e=uv\in f(T)$ with $r(e)> C> k$ for some triangle $T$ in $G$.
    Clearly, there exist $k$ distinct triangles, denoted as $T_1, T_2, \ldots, T_k$, such that $e$ is one of the corresponding edges for each $T_i$ with respect to $f$.
    We will next prove that it is possible to extend $T_1, \ldots, T_k$ into $k$ internally disjoint paths of lengths $p_1, \ldots, p_k$, respectively.
    
    When $p_i = 1$, use the edge $uv$ directly. 
    When $p_i = 2$, let $w_i = V(T_i) \setminus \{u,v\}$ and take the path $uw_iv$.
    For $p_i \geq 3$, we construct a planar graph recursively: Start with triangle $T_i$ and its non-corresponding edge $e_i^*$. 
    By \cref{cl:-12}, $r(e_i^*) \geq C - 1 > k$. 
    Thus there exists a triangle $T_{i_1}$ such that $e_{i}^*$ is a corresponding edge of $T_{i_1}$, with vertex $w_{i_1} = V(T_{i_1}) \setminus V(e_i^*)$ unused in previous constructions.
    
    Repeat this process $p_i - 2$ times: At step $j$, the current triangle $T_{i_j}$ has a non-corresponding edge $e_{i_j}^{*}$ with $r(e_{i_j}^{*}) >|H|$, allowing us to select a new triangle $T_{i_{j+1}}$ where $e_{i_j}^{*}$ is a corresponding edge of $T_{i_{j+1}}$, with new vertex $w_{i_{j+1}} = V(T_{i_{j+1}}) \setminus V(e_{i_{j}}^{*})$ not used in any process.

    After $p_i - 2$ extensions, we obtain a planar graph $B_i$ with faces $T_i, T_{i_1}, \dots, T_{i_{p_i-2}}$. 
    Note that the outer face boundary $\mathcal{C}_i$ of $B_i$ is a cycle of length $p_i + 1$. 
    Removing the edge $uv$ yields a $(u,v)$-path $\mathcal{C}_i - uv$ of length $p_i$. 
    The paths are internally disjoint because each path uses distinct new vertices. 
    These paths form a copy of $H = \Theta(p_1,\dots,p_k)$, contradicting the $H$-freeness of $G$. Thus $r(e) \leq C$ for every edge $e \in E(G)$.
\end{proof}

\begin{figure}[H]
    \centering
\begin{tikzpicture}[scale=0.9]
  \node[inner sep= 1pt](u) at (0,0)[circle,fill]{};
  \node[inner sep= 1pt](u') at (0,-0.3)[]{$u$};
  \node[inner sep= 1pt](u_1) at (0,2)[circle,fill]{};
  \node[inner sep= 1pt](u_2) at (-2,2)[circle,fill]{};
  \node[inner sep= 1pt](u_3) at (-2,4)[circle,fill]{};
  \node[inner sep= 1pt](v) at (2,0)[circle,fill]{};
  \node[inner sep= 1pt](v') at (2,-0.3)[]{$v$};
  \node[inner sep= 1pt](v_1) at (2,2)[circle,fill]{};
  \node[inner sep= 1pt](v_2) at (4,2)[circle,fill]{};
  \node[inner sep= 1pt](v_3) at (4,4)[circle,fill]{};
  \node[inner sep= 1pt](v_4) at (2,4)[circle,fill]{};
  \draw[dashed] (u) -- (v);
  \draw[dashed] (u) -- (u_1);
  \draw[dashed] (u_2) -- (u_1);
  \draw[dashed] (v_1) -- (v);
  \draw[dashed] (v_1) -- (v_2);
  \draw[dashed] (v_4) -- (v_2);
  \draw [red](u) -- (u_2);
  \draw [red](u_2) -- (u_3);
  \draw [red](u_3) -- (v);
  \draw [blue](u) -- (v_1);
  \draw [blue](v_1) -- (v_4);
  \draw [blue](v_3) -- (v_4);
  \draw [blue](v_2) -- (v_3);
  \draw [blue](v) -- (v_2);
  \fill[red, opacity=0.1] (0,0) -- (2,0) -- (0,2) -- cycle;
  \node[inner sep= 1pt](T_1) at (0.4,1)[]{$T_1$};
  \fill[blue, opacity=0.1] (0,0) -- (2,2) -- (2,0) -- cycle;
  \node[inner sep= 1pt](T_t) at (1.6,1)[]{$T_k$};
  \node[inner sep= 1pt](T_t^1) at (2.6,1.3)[]{$T_{k_1}$};
  \node[inner sep= 1pt](T_t^2) at (2.6,2.55)[]{$T_{k_2}$};
  \node[inner sep= 1pt](T_t^3) at (3.5,3.3)[]{$T_{k_3}$};
  \node[inner sep= 1pt](...) at (1,1.7)[]{$\cdots$};
\end{tikzpicture}
\caption{The growth process of a path of length $p_i$ derived from $T_i$.}\label{fig-path}
\end{figure}
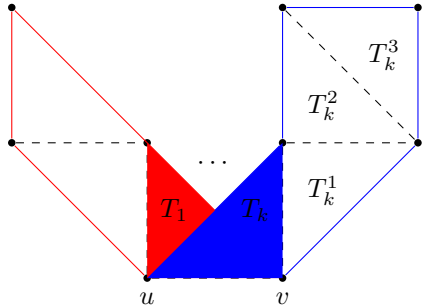

  The following lemma will be used in the proof of \cref{thm:jun2}.

\begin{lemma}\label{cl:color}
 Let $\chi$ be an edge coloring of a complete bipartite graph $G=K_{k^3,k^3}$, then there exists a rainbow star $K_{1,k}$ or a monochromatic matching of size $k+1$ in $G$. 
\end{lemma}

\begin{proof}
  Let $A\cup B$ be the two partition classes of $V(G)$.
  Suppose there are no vertices in $A$ as the center of a rainbow $K_{1,k}$.
  We will show that either there exists a monochromatic $(k+1)$-matching, or a rainbow $K_{1,k}$ with a vertex in $B$ as its center.
  Since each vertex in $A$ is associated with at most $k-1$ colors, for each vertex $v$ in $A$, there exists a monochromatic star $K_{1,k^2}$ centered at $v$.
  Remove all edges except those in these monochromatic stars $K_{1,k^2}$.

  Assume that there is no monochromatic $(k+1)$-matching.
  That is, each color appears in at most $k$ monochromatic $K_{1,k^2}$.
  Otherwise, we can recursively find a monochromatic $(k+1)$-matching.
  Hence, by the pigeonhole principle, there are at least $k^2$ colors appear among these monochromatic $K_{1,k^2}$.
  For each color, we keep only one monochromatic $K_{1,k^2}$, and denote the resulting graph by $G'$.
  As a result, any vertex in $B$ receives edges of distinct colors.
  Then there are $k^4$ edges in $G'$.
  By the pigeonhole principle again, there is a rainbow $K_{1,k}$ with a center in $B$.
\end{proof}

Now we are ready to prove \cref{thm:jun2}.

\begin{proof}[\textbf{Proof of \cref{thm:jun2}}]

    Let $G$ be an $H$-free graph and $\mathcal{K}_3(G)= \{T_1,T_2,\cdots, T_m\}$ be the collection of all triangles in $G$.
    Without loss of generality, we can assume $p_1\le p_2 \le\cdots \le p_k$.     
    Let $t= |V(H)|$ and $\alpha= 1/t^4$.

\medskip
\noindent{\textbf{Case 1.}} $H$ contains at most one triangle.

\vspace{0.2cm}
    By \cref{lem:k-tree}, there exists a constant $C=C(H)$ and a map $f:\mathcal{K}_{3}(G)\rightarrow E(G)$ satisfying:
\begin{itemize}
    \item For each $T\in \mathcal{K}_{3}(G)$, $f(T)$ is an edge in $T$;
    \item For each $e\in \mathcal{K}_{2}(G) = E(G)$, $r(e)=|f^{-1}(e)|\le C$.
    \end{itemize}
    
    For each $T_i \in \mathcal{K}_{3}(G)$, let $x_i = V(T_i) \setminus V(f(T_i))$.
    Let $(A^*,B^*)$ be a random partition of $V(G)$ where each vertex picks each part with probability $1/2$, independently of other vertices.
    For each $T_i$, the probability 
\begin{align*}
   \mathbf{P}(x_i\in A^*\text{ and }V(f(T_i))\subseteq B^*) =1/8, 
\end{align*}    
    so there exists a partition $(A,B)$ such that the number of triangles $T_i$  in $\mathcal{K}_3(G)$ satisfying $x_i\in A$ and $V(f(T_i)) \subseteq B$ is at least $m/8$.
    
    Let $G'$ be the auxiliary graph with vertex set $B$ and edge set $E':=\{f(T_i)\mid  x_i\in A\text{ and }V(f(T_i))\subseteq B \text{ for some } i \in[m] \}$.
    Since for each $e\in E(G)$, $r(e) \le C$, we have $e(G')\ge \frac{m}{8C}$.
    
    If $m= \Omega(n^{2-\alpha})$, by \cref{thm:KST}, $G'$ contains a $K_{t^3,t^3}$ as a subgraph, denote it by $F$.
    Recall that for each $e\in E(G')$, there exists at least one vertex $x$ in $A$ such that $x$ and $e$ form a triangle in $G$. 
    We call one such vertex $x$ the \textit{corresponding vertex} of $e$ in $A$.

\medskip
\noindent\textbf{Subcase 1.1.} There exists a matching $M$ of size $k+1$ in $F$, such that all the edges in $M$ have same corresponding vertex $y$. 
\vspace{0.2cm}

    Suppose $M = \{v_1u_1,\cdots, v_{k+1}u_{k+1}\}$ such that $v_1,\cdots v_{k+1}$ are in the same part of $F$. We will show that we can progressively find $k$ pairwise internally disjoint paths $P_i$ with end vertices $u_1,y$ and length $p_i$ such that each path uses at most one element in $\{v_1,\cdots,v_{k+1},u_2,\cdots,u_{k+1}\}$. When $p_1=1$, we directly choose the edge $yu_1$. For $p_i\ge 2$, the selection depends on the parity of $p_i$ as follows:
   \begin{itemize}
     \item If $p_i$ is even, choose a vertex $v'\in\{v_1,\cdots,v_k\}$ which has not been chosen, yet. Note that we can find a path $L$ with end vertices $u_1,v'$ and length $p_i-1$ in $F$ which is internally disjoint from $V(M)$ and all the paths we have selected. Hence, $Ly$ (the orange line in \cref{1.1}) is a path with end vertices $u_1,y$ and length $p_i$.

     \item If $p_i\ge 3$ and $p_i$ is odd, choose a vertex $u'\in\{u_1,\cdots,u_k\}$ which has not been chosen, yet.  Note that we can find a path $L$ with end vertices $u_1,u'$ and length $p_i-1$ in $F$ which interior disjoint with $V(M)$ and all the paths we have selected. Hence, $Ly$ (the green line in \cref{1.1}) is a path with end vertices $u_1,y$ and length $p_i$.

    \end{itemize}

    Thus, $G$ contains a copy of $H$, a contradiction.

\begin{figure}[!ht]
    \centering
\begin{tikzpicture}[scale=1.2]
\node[inner sep= 1.3pt, blue](y) at (4,3.6)[circle,fill]{};
\node[inner sep= 1.3pt](y1) at (4,3.8)[]{\small $y$};
\node[inner sep= 1.3pt,blue](u) at (3,2)[circle,fill]{};
\node[inner sep= 1.3pt](u1) at (3.25,2.08)[]{\small $u_1$};
\node[inner sep= 1pt](v) at (3,0)[circle,fill]{};
\node[inner sep= 1pt](v1) at (3.2,0)[]{\small $v_1$};
\node[inner sep= 1pt](v2) at (0.6,0)[circle,fill]{};
\node[inner sep= 1pt](v21) at (0.8,0)[]{\small $v_2$};
\node[inner sep= 1pt](u2) at (0.6,2)[circle,fill]{};
\node[inner sep= 1pt](u21) at (0.85,1.9)[]{\small $u_2$};
\node[inner sep= 1pt](v3) at (5.6,0)[circle,fill]{};
\node[inner sep= 1pt](v31) at (5.8,0)[]{\small $v_3$};
\node[inner sep= 1pt](u3) at (5.6,2)[circle,fill]{};
\node[inner sep= 1pt](u31) at (5.8,1.9)[]{\small $u_3$};
\draw[mypink] (v) -- (u);
\draw[mypink] (y) -- (v);
\draw[darkpastelgreen] (u2) -- (y);
\draw[darkpastelgreen] (u2) -- (v2);
\draw[dashed] (y) -- (v2);
\draw[dashed] (y) -- (u3);
\draw[dashed] (v3) -- (u3);
\draw[myorange] (y) -- (v3);
\draw [decorate, decoration=snake, segment length=4mm, myorange](u) -- (v3);
\draw [decorate, decoration=snake, segment length=4mm, darkpastelgreen](u) -- (v2);
\end{tikzpicture}
\caption{The illustration of Subcase 1.1. }\label{1.1}
\end{figure}
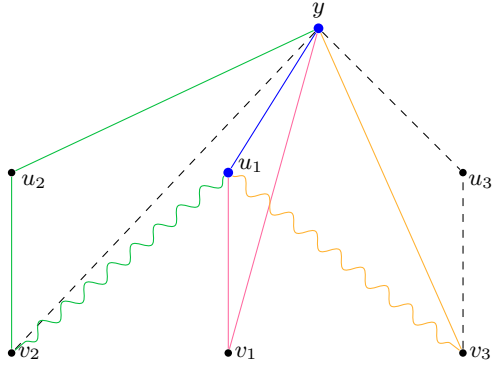

\medskip
\noindent\textbf{Subcase 1.2.} $p_1 =1$, and there does not exist a matching of size $k+1$ in $F$ such that all the $k+1$ edges have same corresponding vertex. 
\vspace{0.2cm}

    Then by the pigeonhole principle, there exists a matching $M$ of size $k$ in $F$ such that each edge in $M$ has a distinct  corresponding vertex.
    Suppose $M = \{v_1u_1,\cdots, v_ku_k\}$ such that $v_1,\cdots v_k$ are in the same part of $F$ and $x_i$ is the corresponding vertex of $u_iv_i$.
    We will show that we can progressively find $k$ pairwise internally disjoint paths $P_i$ with end vertices $u_1,v_1$ and length $p_i$.
    When $p_1=1$, we directly choose the edge $yu_1$. For $p_i\ge 2$, the selection depends on the parity of $p_i$ as follows:
    \begin{itemize}
    \item If $p_i=2$, choose the path $v_1x_1u_1$. Note that we can do this since $H$ contains at most one triangle.

    \item If $p_i\ge 3$ and $p_i$ is odd, choose an edge $u'v'$ in $M$ with $u'\in \{u_2,\cdots,u_k\}$ which has not been chosen yet. Note that we can find a path $L$ with end vertices $u_1,v'$ and length $p_i-2$ in $F$ which is internally disjoint with $V(M)$ and all the paths we have selected. Hence, $Lu'v_1$ (the green line in \cref{1.2}) is a path with end vertices $u_1,v_1$ and length $p_i$.

    \item If $p_i\ge 4$ and $p_i$ is even, choose an edge $u'v'$ in $M$ with $u'\in \{u_2,\cdots,u_k\}$ which has not been chosen, yet. Let $x'$ be the corresponding vertex of $u'v'$ in $A$. Note that we can find a path $L$ with end vertices $u_1,v'$ and length $p_i-3$ in $F$ which is internally disjoint with $V(M)$ and all the paths we have selected. Hence, $Lx'u'v_1$ (the orange line in \cref{1.2}) is a path with end vertices $u_1,v_1$ and length $p_i$.
    \end{itemize}
    
    Thus, $G$ contains a copy of $H$, a contradiction.

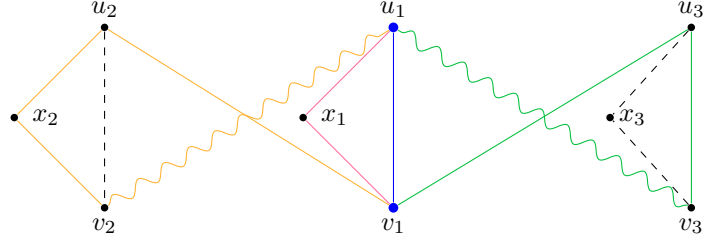
\begin{figure}[!ht]
    \centering
\begin{tikzpicture}[scale=1.2]
  \node[inner sep= 1.3pt,blue](u_1) at (3.2,2)[circle,fill]{};
  \node[inner sep= 1.3pt](u_1') at (3.2,2.2)[]{$u_1$};
  \node[inner sep= 1.3pt,blue](v_1) at (3.2,0)[circle,fill]{};
  \node[inner sep= 1.3pt](v_1') at (3.2,-0.2)[]{$v_1$};
  \node[inner sep= 1pt](x_1) at (2.2,1)[circle,fill]{};
  \node[inner sep= 1pt](x_1') at (2.55,1)[]{$x_1$};
  \draw[mypink](x_1) -- (u_1);
  \draw[mypink](x_1) -- (v_1);
  \draw[blue](u_1) -- (v_1);
  \node[inner sep= 1pt](u_2) at (0,2)[circle,fill]{};
  \node[inner sep= 1pt](u_2') at (0,2.2)[]{$u_2$};
  \node[inner sep= 1pt](v_2) at (0,0)[circle,fill]{};
  \node[inner sep= 1pt](v_2') at (0,-0.2)[]{$v_2$};
  \node[inner sep= 1pt](x_2) at (2-3,1)[circle,fill]{};
  \node[inner sep= 1pt](x_2') at (2.35-3,1)[]{$x_2$};
  \draw[myorange](x_2) -- (u_2);
  \draw[myorange](x_2) -- (v_2);
  \draw[myorange](u_2) -- (v_1);
  \draw [decorate, decoration=snake, segment length=4mm,myorange](v_2) -- (u_1);
  \node[inner sep= 1pt](u_3) at (6.5,2)[circle,fill]{};
  \node[inner sep= 1pt](u_3') at (6.5,2.2)[]{$u_3$};
  \node[inner sep= 1pt](v_3) at (6.5,0)[circle,fill]{};
  \node[inner sep= 1pt](v_3') at (6.5,-0.2)[]{$v_3$};
  \node[inner sep= 1pt](x_3) at (2+3.6,1)[circle,fill]{};
  \node[inner sep= 1pt](x_3') at (5.85,1)[]{$x_3$};
  \draw[dashed] (u_2) -- (v_2);
  \draw[dashed] (u_3) -- (x_3);
  \draw[dashed] (v_3) -- (x_3);
   \draw[darkpastelgreen](u_3) -- (v_3);
    \draw[darkpastelgreen](u_3) -- (v_1);
  \draw [decorate, decoration=snake, segment length=4mm, darkpastelgreen](v_3) -- (u_1);
\end{tikzpicture}
\caption{The illustration of Subcase 1.2.}\label{1.2}
\end{figure}

\noindent\textbf{Subcase 1.3.} $p_1 >1$, and there does not exist a matching of size $k+1$ in $F$ such that the $k+1$ edges in $F$ have same corresponding vertex.

    By \cref{cl:color}, there exists a copy of $K_{1,k}$, denote it by $N = \{uv_1,\cdots, uv_k\}$, such that each edge in $N$ has a distinct corresponding vertex. Let $x_i$ be the corresponding vertex of  $uv_i$.
    Choose a vertex $u'$ in $F$ such that $u,u'$ are in the same part. 
    Since $p_1 >1$, we have $p_i\ge 2$ for all $i\in[k]$.
    We will show that we can progressively find $k$ pairwise internally disjoint paths $P_i$ with end vertices $u,u'$ and length $p_i$ such that each path uses at most one element in $\{v_1,\cdots,v_{k}\}$.
    
    \begin{itemize}

    \item If $p_i$ is odd, choose a vertex $v'\in\{v_1,\cdots,v_k\}$ which has not been chosen, yet.
    Let $x'$ be the corresponding vertex of $uv'$.
    Note that we can find a path $L$ with end vertices $u',v'$ and length $p_i-2$ in $F$ which is internally disjoint with $V(N)$ and all the paths we have selected. Hence, $Lx'u$ (the green line in \cref{1.3}) is a path with end vertices $u,u'$ and length $p_i$.

    \item If $p_i$ is even, choose a vertex $v'\in\{v_1,\cdots,v_k\}$ which has not been chosen, yet.  Note that we can find a path $L$ with end vertices $u',v'$ and length $p_i-1$ in $F$ which is internally disjoint with $V(N)$ and all the paths we have selected. Hence, $Lu$ (the orange line in \cref{1.3}) is a path with end vertices $u,u'$ and length $p_i$.
    \end{itemize}

    Thus, $G$ contains a copy of $H$, a contradiction.

    \begin{figure}[!ht]
    \centering
\begin{tikzpicture}[scale=1.4]
  \node[inner sep= 1.3pt,blue](u) at (2.5,3)[circle,fill]{};
  \node[inner sep= 1.3pt](u') at (2.5,3.2)[]{$u$};
  \node[inner sep= 1.3pt,blue](u') at (1,3)[circle,fill]{};
  \node[inner sep= 1.3pt](u1') at (0.82,3.08)[]{$u'$};
  \node[inner sep= 1pt](v_1) at (1,1)[circle,fill]{};
  \node[inner sep= 1pt](v_1') at (1,0.8)[]{$v_1$};
   \node[inner sep= 1pt](v_2) at (2.5,1)[circle,fill]{};
  \node[inner sep= 1pt](v_2') at (2.5,0.8)[]{$v_2$};
   \node[inner sep= 1pt](v_3) at (4,1)[circle,fill]{};
  \node[inner sep= 1pt](v_3') at (4,0.8)[]{$v_3$};
  \node[inner sep= 1pt](x) at (0,3.4)[circle,fill]{};
  \node[inner sep= 1pt](x1) at (-0.1,3.2)[]{$x_1$};
  \draw[dashed](u) -- (v_1);
  \draw[darkpastelgreen](u) -- (x);
  \draw[darkpastelgreen](v_1) -- (x);
  \draw[myorange](v_2) -- (u);
  \draw[myorange](v_3) -- (u);
   \draw [decorate, decoration=snake, segment length=4mm,darkpastelgreen](u') -- (v_1);
  \draw [decorate, decoration=snake, segment length=4mm,myorange](v_2) -- (u');
  \draw [decorate, decoration=snake, segment length=4mm,myorange](v_3) -- (u');
\end{tikzpicture}
\caption{The illustration of Subcase 1.3.}\label{1.3}
\end{figure}
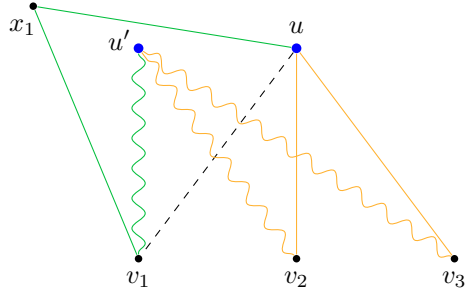

\medskip
\noindent\textbf{Case 2.} $H$ contains at least two triangles and does not contain $\Theta(1,2,2,3)$ as a subgraph.
\vspace{0.2cm}

This means $p_1 =1$, $p_2 = p_3 =2$ and $p_i\ge 4$ for $i\ge 4$. 
Since $H$ contains $B_2$ as a subgraph, by (\ref{thm:book}), we have $\ex(n, K_3,H)\ge \ex(n, K_3,B_2) \ge n^{2-o(1)}$.
For the upper bound, assume to the contrary that there exists $\varepsilon >0$ such that $|\mathcal{K}_{3}(G)|\ge \varepsilon n^2$.

By \cref{lem:theta graph}, there exists a map $f:\mathcal{K}_{3}(G)\rightarrow\binom{E(G)}{2}$ and a constant $C=C(H)$ such that

    \begin{itemize}
    \item For each $T\in \mathcal{K}_{3}(G)$, $f(T)$ is a set of two different edges in $T$;
    \item For each $e\in E(G)$, $r(e)\le C$, where $r(e)$ is the number of the triangles $T$ such that $e\in f(T)$.
    \end{itemize}
    
Define
\[\mathcal{T}_{1}(G)= \{T\mid T\in \mathcal{K}_{3}(G) \text{ and there exists some } e\in T \text{ with } d_G(e) \ge C\}; \quad \mathcal{T}_{2}(G)=\mathcal{K}_{3}(G)\setminus\mathcal{T}_{1}(G),\]
where $d_G(e)$ is the number of triangles in $G$ contains edge $e$.

Suppose $\mathcal{T}_{1}(G) = \{T_1,\cdots,T_m\}$.
Since there is no book with $C$ page vertices in $\mathcal{T}_{2}(G)$, we have $|\mathcal{T}_{2}(G)| = o(n^2)$ from (\ref{thm:book}), which implies that $m\ge  |\mathcal{K}_{3}(G)|/2$.

For each triangle $T_i$, suppose $f(T_i)=\{e_i,e_i'\}$ and $V(T_i)=\{x_i,y_i,z_i\}$ with $e_i = x_iy_i$, $e_i' = x_iz_i$, $e_i^*=y_iz_i.$

 Let $(X^*,Y^*,Z^*)$ be a random partition of $V(G)$ where each vertex picks each part with probability $1/3$, independently of other vertices.
    For each $T_i$, the probability 
\begin{align*}
   \mathbf{P}(x_i\in X^*, ~y_i\in Y^*, ~z_i \in Z^*) = 1/27.
\end{align*}    
    Hence there exists a partition $(X,Y,Z)$ such that the number of triangles $T_i$  in $\mathcal{K}_3(G)$ satisfying $x_i\in X$, $y_i\in Y$, $z_i\in Z$ is at least $m/27$. Let 
\[\mathcal{T}_{3}(G)= \{T_i\mid T_i\in \mathcal{T}_{1}(G), ~x_i\in X, ~y_i\in Y, ~z_i \in Z\}.\]

Without loss of generality, assume that $\mathcal{T}_{3}(G) = \{T_1,\cdots ,T_{m'}\}$,  then $m' \ge m/27$. 

Let $E'$ be the collection of edges $e$ satisfying $e\in T$ and $d_G(e)\ge C$ for some triangle $T$.
We claim that $|E'| =O(n^{2-\alpha})$.
Otherwise, by \cref{thm:KST}, there exists a copy of $K_{t^3,t^3}$ in $E'$, and it is easy to obtain a copy of $H$ in $G$ by the definition of $E'$.

Let 
\[\mathcal{T}_{4}(G)= \{T_i\mid T_i\in \mathcal{T}_{3}(G), ~d_G(e_i^*)\ge C\}  \text{\quad and \quad} \mathcal{T}_{5}(G)=\mathcal{T}_{3}(G)\setminus\mathcal{T}_{4}(G).\]

Then for any $T\in\mathcal{T}_{5}(G)$, there exists an edge $e \in f(T)$ such that $d_G(e)\ge C$. Hence, we can find a function $g:\mathcal{T}_{5}(G)\rightarrow E'$ such that for any triangle $T\in\mathcal{T}_{5}(G)$, $g(T)\in f(T)$. Moreover, for any $e\in E'$, $|g^{-1}(e)|\le C$, which means $$|\mathcal{T}_{5}(G)|\le C|E'|=O(n^{2-\alpha}).$$

Hence, $\mathcal{T}_{4}(G)\ge m'/2$.
Without loss of generality, we can assume $\mathcal{T}_{4}(G) = \{T_1, T_2 \cdots, T_{m''}\}$ with $m''\ge \frac{m'}{2} \ge \frac{m}{54} \ge \frac{\varepsilon n^2}{54}$. 
Let $E^* = \{e^*_i\mid i\in [m'']\}$. Then we have $|E^*| \le |E'|= O(n^{2-\alpha})$.
Let $G^*$ be the bipartite graph between $X$ and $Y\cup Z$ with the edge set $\bigcup_{T\in \mathcal{T}_{4}(G)} f(T)$.

\begin{claim}\label{cl:3.8}
There exists a copy of complete bipartite graph $F=K_{t^3,t^3}$ in $G^*$, such that $V(F) \cap (Y\cup Z)$ contains two vertices forming an edge in $E^*$.
\end{claim}
\begin{proof}
Suppose, for a contradiction, the statement does not hold. 
To derive a contradiction, we use double counting to count the number of tuples $(x,w_1,w_2,\dots, w_{t^3})$ satisfying the following conditions:
\begin{itemize}
    \item $x\in X$, and for $i\in[t^3]$, the $w_i$'s are different vertices in $Y\cup Z$;
    \item $w_1w_2\in E^*$;
    \item $xw_i \in E(G^*)$.
\end{itemize} 

Let $s$ denote the number of such tuples.
There are at most $2|E^*|$ choices for the pair $(w_1, w_2)$ and at most $n^{t^3-2}$ choices for the tuple $(w_3, w_4, \dots, w_{t^3})$.
Furthermore, for each fixed choice of $(w_1, w_2, \dots, w_{t^3})$, by the assumption that no copy of $K_{t^3,t^3}$ satisfies the condition in \cref{cl:3.8}, there are at most $t^3 - 1$ choices for $x$.
Hence 
$$ s\le (t^3-1)\cdot 2|E^*| \cdot n^{t^3-2}  = O(n^{t^3 -\alpha}). $$

Next we estimate the lower bound of $s$.
Recall that each triangle in $\mathcal{T}_{4}(G)$ has exactly one vertex in each of the sets  $X,Y$ and $Z$.
Let $X'$ be the set of vertices $x$ in $X$ such that $x$ is contained in more than $\frac{m''}{3n}$ triangles in $\mathcal{T}_4(G)$.
Then, for any fixed $x\in X'$, there are at least $\frac{2m''}{3n}$ choices for the pair $(w_1, w_2)$, where $(x,w_1, w_2)$ forms a triangle in $\mathcal{T}_{4}(G)$ and $w_1w_2 \in E^*$.
It follows form  \cref{lem:theta graph} that for any edge $e \in E(G^*)$, we have $r(e) \le C$.
Therefore, for  every $x \in X'$, the number of triangles in $\mathcal{T}_4(G)$ containing $x$ is at most $C d_{G^*}(x)<Cn$.
Then,
 $$m'' =|\mathcal{T}_{4}(G)| \le |X'|Cn + n\cdot \frac{m''}{3n}, $$
which implies $|X'|\ge \frac{m''}{2Cn}$.

Since the number of triangles in $\mathcal{T}_4(G)$ containing $x$ is at most $C d_{G^*}(x)$, we derive that $d_{G^*}(x) \geq \frac{m''}{3Cn}$ for every $x\in X'$.
Then, the vertices $w_3, \dots, w_{t^3}$ can be chosen from $N_{G^*}(x)$ progressively. 
Recall that $m''\ge \frac{\varepsilon n^2}{54}$. 
Thus,
$$ s\ge \frac{m''}{2Cn} \cdot \frac{m''}{3n}\cdot \binom{\frac{m''}{3Cn}-2}{t^3-2} =\Omega\left(\left(\frac{m''}{n}\right)^{t^3}\right)= \Omega(n^{t^3}),$$
which contradicts the earlier bound. This completes the proof of the claim.
\end{proof}

By \cref{cl:3.8}, there exists a copy of complete bipartite graph $K_{t^3,t^3}$ in $G^*$, denoted by $F$, such that $V(F)$ contains two vertices $y^* \in Y$ and $z^* \in Z$ with $y^*z^* \in E^*$.
For any edge $e$ in $F$, there exists a triangle $T$ in $G$ that contains $e$.
Let the vertex $V(T) \setminus e$ be the corresponding vertex for this edge.

\begin{claim}\label{cl:504}
There exists a matching $M$ in $F$ of size $k+1$ with $y^*,z^*\notin V(M)$ such that all edges in $M$ have the same corresponding vertex or each edge has a distinct corresponding vertex in $V_0=V(G^*) \setminus (V(M) \cup \{y^*,z^*\})$.
\end{claim}

\begin{proof}

Let $M$ be a maximal matching in $F$ with $y^*,z^*\notin V(M)$ such that each edge has a distinct corresponding vertex in $V_0$. 
Then the corresponding vertex of each edge in $V_0\cap V(F)$ will belong to $V(M) \cup \{y^*,z^*\}$ or equal to the corresponding vertex of some edges in $M$.
If $|M| \ge k+1$, then we are done.
Suppose $|M| \le  k$, then by the pigeonhole principle, there exists a vertex such that it serves as the corresponding vertex for at least $k+1$ independent edges in $F$.
\end{proof}

By \cref{cl:504}, we distinguish two cases:

\medskip

\noindent\textbf{Subcase 2.1.} There exists a matching $M$ in $F$ of size $k+1$ such that all edges in $M$ have the same corresponding vertex.
\vspace{0.2cm}

The proof of this case is same as the the proof of Subcase 1.1.

\vspace{0.2cm}
\noindent\textbf{Subcase 2.2.} There exists a matching $M$ in $F$ of size $k+1$ such that each edge has a distinct corresponding vertex in $V_0$.
\vspace{0.2cm}

    Suppose $M' = \{v_1u_1,\cdots, v_{k}u_{k}\}\subseteq M$ such that $u_1,\cdots u_{k},z^*,y^*$ are in the same part of $F$. 
    Let $x_i$ be the corresponding vertices of $v_iu_i$ for $i\in [k]$.
    We will show that we can progressively find $k$ pairwise internally disjoint paths $P_i$ with end vertices $y^*,z^*$ and length $p_i$.
    Let $V_1= \{x_1,x_2,\cdots x_k\} \cup V(M)$.
    Recall that we have $p_1 = 1, p_2=p_3 =2$ and $p_i \ge 2, p_i \ne 3$ for all $i\ge 2$.
    \begin{itemize}
    \item If $p_i=1$, choose the edge $y^*z^*$. 
    
    \item If $p_i$ is even, we can find a path with end vertices $y^*,z^*$ and length $p_i$ in $F$ which is internally disjoint with $V_1$ and all the paths we have selected.
     
    \item If $p_i\ge 5$ and $p_i$ is odd, choose an edge $u'v'\in M, v'\in \{v_1,\cdots v_{k}\}$ which has not been chosen, yet. 
    Let $x'\in\{x_1,x_2,\cdots x_k\}$ be the corresponding vertex of $v'u'$.
    Note that we can find a path $L$ with end vertices $z^*,u'$ and length $p_i-3$ in $F$ which is internally disjoint with $V(M)$ and all the paths we have selected. Hence, $Lx'v'y^*$ is a path with end vertices $z^*,y^*$ and length $p_i$.
    \end{itemize}

    Thus, $G$ contains a copy of $H$, a contradiction.

    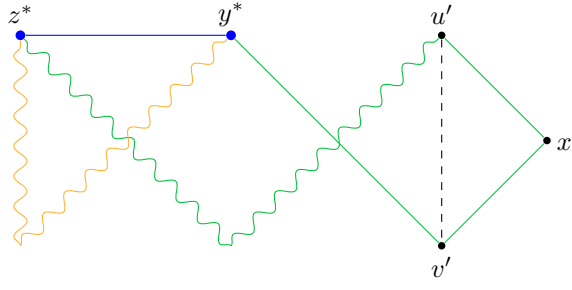
\begin{figure}[!ht]
    \centering
\begin{tikzpicture}[scale=1.4]
  \node[inner sep= 1.3pt,blue](z) at (0,2)[circle,fill]{};
  \node[inner sep= 1.3pt](z') at (0,2.2)[]{$z^*$};
  \node[inner sep= 1.3pt,blue](y) at (2,2)[circle,fill]{};
  \node[inner sep= 1.3pt](y') at (2,2.2)[]{$y^*$};
    \node[inner sep= 1pt](v) at (4,0)[circle,fill]{};
 \node[inner sep= 1pt](v_1') at (4,-0.2)[]{$v'$};
%   \node[inner sep= 1pt](v3) at (2,0)[]{};
%    \node[inner sep= 1pt](v4) at (0,0)[]{};
   \node[inner sep= 1pt](u) at (4,2)[circle,fill]{};
  \node[inner sep= 1pt](u1) at (4,2.2)[]{$u'$};
 \node[inner sep= 1pt](x) at (5,1)[circle,fill]{};
  \node[inner sep= 1pt](x1) at (5.2,1)[]{$x'$};
  \draw[blue](z)--(y);
  \draw[dashed](u)--(v);
  \draw [darkpastelgreen](x) -- (u);
  \draw [darkpastelgreen](x) -- (v);
   \draw [darkpastelgreen](y) -- (v);
  \draw [decorate, decoration=snake, segment length=4mm,darkpastelgreen](2,0) -- (z);
  \draw [decorate, decoration=snake, segment length=4mm,darkpastelgreen](2,0) -- (u);
   \draw [decorate, decoration=snake, segment length=4mm,myorange](0,0) -- (z);
  \draw [decorate, decoration=snake, segment length=4mm,myorange](0,0) -- (y);
  
\end{tikzpicture}
\caption{The illustration of subcase 2.2.}
\end{figure}

\vspace{0.2cm}
\noindent\textbf{Case 3.} $H$ contains $\Theta(1,2,2,3)$ as a subgraph.
\vspace{0.2cm}

By \cref{cor:theta}, we have $\mathrm{ex}(n, K_3, H) = O(n^2)$. Moreover, let $G_0$ be the graph obtained from $K_{\lfloor \frac{n}{2}\rfloor,\lceil\frac{n}{2}\rceil}$ by adding a maximal matching in one part.
It is easy to check that $G_0$ is $H$-free, and the number of triangles in $G_0$ is at least $2\lfloor 
 n/4\rfloor^2$, which implies $\mathrm{ex}(n, K_3, H) = \Omega (n^2)$.

This finishes the proof of \cref{thm:jun2}. \end{proof}

\section{Generalized Tur\'an number of disjoint subgraphs}

In this section, we will prove \cref{thm:main}. First we prove the following stability lemma for edge-critical graphs with $\chi(F)=3$.

\begin{lemma}\label{lmm:edge-critical}
  Let $F$ be an edge-critical graph with $\chi(F)=3$.
  For any $\varepsilon>0$, there exists $n_0=n_0(F, \varepsilon) \in \mathbb{N}$ such that the following holds.
  If $G$ is an $F$-free graph on $n \geq n_0$ vertices with $e(G)\geq n^2/4-\varepsilon^2 n^2$, then there exists an induced subgraph $G^\star \subseteq G$ with $|G^\star|\ge (1-2\varepsilon)n$, such that $G^\star$ is a bipartite graph with bipartition $A\cup B$ satisfying $|A|,|B|\geq \delta(G^\star)\ge  \frac{(1-2\varepsilon)^2}{2}n$.
\end{lemma}

\begin{proof}
  Assume $1/2 \ge \varepsilon  > 0$  and let $G_n=G$. 
  Assuming $G_j$ is defined, if there exists $v_j\in V(G_j)$ with $d_{G_j}(v_j)<(1/2-\varepsilon) j$, let $G_{j-1}=G_j-\{v_j\}$.
  This process concludes at $G_t$ with $\delta(G_t)\geq (1/2-\varepsilon) t$ and $|G_t|=t$.
  We will now demonstrate that $t\geq (1-2\varepsilon)n$.
  Let $m=(1-2\varepsilon)n$.

  In the case where $t<m$, we have
\begin{align*}
  e(G_m) & = e(G)-\sum_{j=m+1}^n d_{G_j}(v_j) >\frac{n^2}{4}-\varepsilon^2 n^2- \sum_{j=m+1}^n \left(\frac{1}{2}-\varepsilon\right)j \\
  & >\frac{n^2}{4}-\varepsilon^2 n^2 - (n-m) \left(\frac{1}{2}-\varepsilon\right)n=\frac{n^2}{4}-\varepsilon^2 n^2 - 2\varepsilon n\left(\frac{1}{2}-\varepsilon\right)n\\
  & = \frac{(1-4\varepsilon+4\varepsilon^2)n^2}{4} = \frac{m^2}{4}.
\end{align*}
  \cref{thm:simonovits} implies that $G_m\subseteq G$ contains $F$, a contradiction.

  Since $\delta(G_t)\geq (1/2-\varepsilon)t$, we have
\begin{align}\label{Gt}
  e(G_t)\ge \frac{(1/2-\varepsilon)t^2}{2}\ge (\frac{1}{2}-\varepsilon)\binom{t}{2}.\tag{$\spadesuit$}
\end{align}
  Then there is a balanced bipartition $V(G_t)=V_1\cup V_2$ with $e(V_1)+e(V_2)\le 8\sqrt{\varepsilon} t^2$ from \cref{Erdos-Simon}.

  Let $A\cup B$ be a bipartition of $G_t$ that minimizes $e(A)+e(B)$.
  Then $e(A)+e(B)\le 8\sqrt{\varepsilon} t^2$.
  Furthermore, we have 
\begin{align*}
  (1/2-3\sqrt[4]{\varepsilon})t\le |A|,|B| \le (1/2+3\sqrt[4]{\varepsilon})t.
\end{align*}
  Otherwise,
\begin{align*}
  e(G_t)\le |A||B|+e(A)+e(B) < t^2/4-9\sqrt{\varepsilon} t^2+ 8\sqrt{\varepsilon} t^2<t^2/4- \varepsilon t^2/2.
\end{align*}
  This  contradicts  (\ref{Gt}).

  Let $A'=\{v\in A:~|N(v)\cap A|\ge \sqrt[4]{100\varepsilon} t\}$ and $B'=\{v\in B:~|N(v)\cap B|\ge \sqrt[4]{100\varepsilon} t\}$.
  
\begin{claim}\label{cl:emptyAB}
    $A'=B'=\varnothing$.
\end{claim}

\begin{proof}
  Assume that $v\in A'$.
  By the definition of $A$, $|N(v)\cap B|\ge |N(v)\cap A|\ge \sqrt[4]{100\varepsilon}t$.  
  Let $K_{a,b}^{+}$ be the graph obtained by adding one edge to $K_{a,b}$.
  Since every edge-critical graph $F$ with $\chi(F)=3$ is a subgraph of $K_{|F|,|F|}^{+}$, $N(v)$ is $K_{|F|,|F|}$-free.
  Therefore, by \cref{thm:KST}, the number of edges between $N(v)\cap A$ and $N(v)\cap B$ is $o(t^2)$.
  That is, the number of missing edges between $N(v)\cap A$ and $N(v)\cap B$ is at least $\sqrt{100\varepsilon} t^2-o(t^2)>9\sqrt{\varepsilon} t^2$.
  Then
\begin{align*}
  e(G_t)< |A||B|+e(A)+e(B)-9\sqrt{\varepsilon} t^2\le t^2/4-\sqrt{\varepsilon} t^2<t^2/4- \varepsilon t^2/2.
\end{align*}
 This contradicts (\ref{Gt}).
\end{proof}
  By \cref{cl:emptyAB}, for any vertex $v\in A$, $|N(v)\cap A|< \sqrt[4]{100\varepsilon} t$.
  If $e(A)>0$, choose a subset $T\subseteq A$ with $e(T)>0$ and $|T|=|F|$.
  For any $v\in T$, $v$ is non-adjacent to at most 
\begin{align*}
  |B|-(\delta(G_t)-\sqrt[4]{100\varepsilon} t)\le \sqrt[4]{200\varepsilon}t
\end{align*}
  vertices in $B$.
  Let 
  \[T'= \left(\bigcap_{v\in T}N(v)\right)\cap B.\]
  Then $|T'|\geq |B|-|F|\sqrt[4]{200\varepsilon}t>|F|$, which means $G[T\cup T']$ contains $K^+_{|F|,|F|}$ as subgraph, a contradiction.
  Thus, $A$ is an independent set.
  Similarly, $B$ is an independent set.
  Moreover, $|A|,|B|\ge \delta(G_t)\ge \frac{(1-2\varepsilon)^2}{2}n$.
\end{proof}

  We denote by $\Theta(p_0;p_1,\ldots,p_{t-1})$ an edge-critical generalized theta graph such that $p_0$ and $p_i$ have different parity for all $i\in[t-1]$.
  Suppose $n$ is sufficiently large. 
  For any edge-critical generalized theta graph $F$, \cref{thm:simonovits} implies that for any $F$-free graph $G$ on $n$ vertices, we have $e(G)\le \lfloor n^2/4\rfloor$.
  The following lemma provides a stronger result and plays a crucial role in the proof of \cref{thm:main}.

\begin{lemma}\label{lmm:theta}
  Let $t\geq 2$ and $F=\Theta(p_0;p_1,\ldots,p_{t-1})$ be an edge-critical generalized theta graph.
  Let $c$ be a constant and $n \ge n_0=n_0(F,c)$ sufficiently large. 
  Let $r\geq 3$ and $G$ be an $F$-free graph on $n$ vertices, then
\begin{align*}
  e(G)+c|\mathcal{K}_r(G)| \leq\lfloor n^2/4\rfloor,  
\end{align*}
and the equality holds if and only if $G=T_2(n)$.
\end{lemma}

\begin{proof}
 By \cref{lem:k-tree}, we have $|\mathcal{K}_{r+1}(G)| \le C_F \cdot  |\mathcal{K}_r(G)|$ for some constant $C_F $, so it is sufficient to show that \cref{lmm:theta} holds for $r=3$.
Define $\phi(G)=e(G)+c|\mathcal{K}_3(G)|.$ 
Let $G$ be an $F$-free graph with $n$ vertices that maximizes $\phi(G)$.
Since $T_2(n)$ is $F$-free,  $\phi(G)\geq \phi(T_2(n))=\lfloor n^2/4\rfloor$.
Let $\varepsilon>0$ be sufficiently small.
Since every edge-critical generalized theta graph does not contain $\Theta(1,2,2,3)$ as a subgraph, by \cref{thm:jun2}, we have $|\mathcal{K}_3(G)|=o(n^2)$, implying
\begin{align*}
  e(G)=\phi(G)-c|\mathcal{K}_3(G)|>n^2/4-\varepsilon^2 n^2
\end{align*} when $n$ sufficiently large.

By \cref{lmm:edge-critical}, there exists an induced subgraph $G^\star \subseteq G$ with $|G^\star|\ge (1-2\varepsilon)n$, such that $G^\star$ is a bipartite graph with  bipartition $A\cup B$ satisfying $|A|,|B|\ge \delta(G^\star)\ge \frac{(1-2\varepsilon)^2}{2}n$.

Let $T$ be the tree obtained by deleting a root vertex of $F$. By \cref{cor:tree}, there is a constant $C_0$ such that $\ex(n,K_{r-1},T)\leq C_0n$. The following claim shows that there is no vertex with few neighbours in $G$.

\begin{claim}\label{cl:linear-deg}
Let $q=1+c\cdot C_0$.
For any $v\in V(G)$, $d_G(v)\geq \frac{n}{2q+1}$.
\end{claim}

\begin{proof}
For any vertex $v\in V(G^\star)$, $d_G(v)\geq \frac{(1-2\varepsilon)^2}{2}n$.
Now consider a vertex $v\in V(G)\backslash V(G^\star)$.
Since $G$ is $F$-free, $G[N(v)]$ is $T$-free.
By \cref{cor:tree}, the number of edges in $G[N(v)]$ is at most $C_0\cdot d_G(v)$.
Define $d_3(G,v)$ as the number of triangles containing $v$ and $d_G^*(v)=d_G(v)+cd_3(G,v)$, then 
\begin{align*}
  d^*_{G}(v)\leq d_G(v)+c\cdot C_0d_G(v)=qd_G(v).
\end{align*}

  If $d(v)<n/(2q+1)$, we can construct an $F$-free graph $G'$ with  $\phi(G')>\phi(G)$, leading to a contradiction.
  Delete $v$ from $G$, and add a new vertex $v'$ and edges between $v'$ and all the vertices in $B$.
  Let $G'$ be the resulting graph.
  Note that 
\begin{align*}
d^*_{G}(v)\leq qd_G(v)<\frac{qn}{2q+1}<|B|.
\end{align*}
  Thus, we have
\begin{align*}
  \phi(G')= \phi(G)-d_G^*(v)+|B|>\phi(G).
\end{align*}

We now show that $G'$ is $F$-free, which contradicts the maximality of $G$.
Suppose that $G'$ contains a copy of $F$, say $F^*$, then $F^*$ must contain $v'$.
Let $w_1,\ldots,w_b$ be the neighbours of $v'$ in $F^*$.
Then $\{w_1,\ldots,w_b\}$ is a subset of $B$, hence $|N(w_i)\cap A|>(1/2-2\varepsilon)n$ for each  $i\in[b]$.
That is, we can select a vertex $v''\in N(w_1)\cap\cdots\cap N(w_b)\cap (A\backslash V(F^*))$ to replace $v'$, which means we can find a copy of $F$ in $G'-\{v'\}\subseteq G$.
Hence we are done.
\end{proof}

The next claim shows that we can not find a large book in $G$.

\begin{claim}
 $G$ is $B_{s}$-free, where $s=|F|(2q+2)^{|F|}$.
\end{claim}

\begin{proof}
  Suppose $G$ contains a copy of $B_s$ and let $S$ be the set of page vertices of $B_s$.
  Consider the subset $R \subseteq A \cup B$ chosen uniformly at random with $|R|=|F|$. 
  Since $|A\cup B| \ge (1-2\varepsilon) n$ and $|N(v)|\ge \frac{n}{2q+1}$ by \cref{cl:linear-deg}, we derive that $|N(v) \cap  (A\cup B)| \ge (\frac{1}{2q+1}-2\varepsilon)n$.
  So the probability that $v$ belongs to $N(R)$, the common neighborhood of $R$, is
  $$\mathbf{P}[v\in N(R)]\ge \binom{(\frac{1}{2q+1}-2\varepsilon)n}{|F|}/\binom{n}{|F|}\ge (2q+2)^{-|F|}.$$
  Let $X$ be the random variable representing the number of common neighbours of $R$ in $S$. 
  By the linearity of expectation,
  $$\mathbf{E}[X]\ge s(2q+2)^{-|F|}= |F|.$$
  Thus there exists a subset $R$ that has $|F|$ common neighbours in $S$.
  The book $B_{|F|}\subset B_{s}$ and $R$ together form a copy of $K_{|F|,|F|+2}^+$, which contains every edge-critical graph on $|F|$ vertices, a contradiction.
\end{proof}

Now we can also show that there is no triangle in $G$.

\begin{claim}\label{cl:xiaowoniufans}
$G$ is triangle-free.
\end{claim}

\begin{proof}
  We first estimate the sum of the degrees of the vertices in a triangle.
  Since $G$ is $B_{s}$-free, for any edge $uv\in G$, $|N(u)\cap N(v)|< s$. 
  Thus, for any triangle $uvw\in G$, we have
  \[d(u)+d(v)+d(w)\le n+|N(u)\cap N(v)|+|N(u)\cap N(w)|+|N(v)\cap N(w)|\le n+3s. \tag{$\clubsuit$}\]

  Let $L$ be the set of vertices $v$ in $V(G)\backslash V(G^\star)$ such that $d(v)\le (\frac{1}{2}-\frac{1}{8q+4})n$. 
  Let $\ell=|L|$ and $G'=G-L$. Since $G'$ is $B_{s}$-free, we have $e(G')\le \frac{1}{4}(n-\ell)^2$.
  Thus
\begin{align*}
e(G)&\le e(G')+(\frac{1}{2}-\frac{1}{8q+4})n\ell\\
&\le  \frac{1}{4}(n-\ell)^2+(\frac{1}{2}-\frac{1}{8q+4})n\ell \quad (\text{by \cref{thm:simonovits}} )\\
&\le \frac{1}{4}n^2+(\varepsilon-\frac{1}{8q+4})n\ell. \quad(\ell\le |V(G)\backslash V(G^\star)|\le 2\varepsilon n)
\end{align*}
 Moreover, for any triangle $uvw$ in $G$, at least two of $u,v,w$ belong to $L$. Indeed, assume that $u,v\notin L$,  by \cref{cl:linear-deg},
\[d(u)+d(v)+d(w)\ge 2(\frac{1}{2}-\frac{1}{8q+4})n+\frac{n}{2q+1}\ge (1+\frac{1}{4q+2})n>n+3s,\]
a contradiction to $(\clubsuit)$, where the last inequality holds since $n$ is sufficiently large. Assume $u,v\in L$. Note that $uvw$ is a triangle, $u$ and $v$ are adjacent. Since $G$ is $B_{s}$-free, $|N(u)\cap N(v)|< s$.
Thus, the number of triangles which contain edge $uv$ is at most $s-1$. Then
\begin{align*}
  |\mathcal{K}_3(G)|\le (s-1)\binom{\ell}{2}.
\end{align*}
Because $G$ maximizes the function $\phi(G)= e(G)+c|\mathcal{K}_3(G)|$, 

\[ \lfloor n^2/4\rfloor \le \phi(G)\le \frac{1}{4}n^2+\left(\varepsilon-\frac{1}{8q+4}\right)n\ell+ \frac{s\ell^2}{2},\]
which implies that

\[-1\le\left((1+s)\varepsilon-\frac{1}{8q+4}\right)n\ell,  \quad(\ell\le |V(G)\backslash V(G^\star)|\le 2\varepsilon n)\]
so $\ell=0$. Hence, $G$ is triangle-free.
\end{proof}

 Thus, $\phi(G)=e(G)\leq \lfloor n^2/4\rfloor$, and the equality holds if and only if $G=T_2(n)$ by Mantel's Theorem.
\end{proof}

Now we can start to prove the  \cref{thm:main}.

\begin{proof}[\textbf{Proof of \cref{thm:main}}]
Let $G$ be a $kF$-free graph with $n$ vertices that maximizes $|\mathcal{K}_r(G)|$.

\medskip
\noindent{\textbf{Case 1.}} $3\le r\le k+1$.

\vspace{0.2cm}

  Since $K_{k-1}\vee T_2(n-k+1)$ is $kF$-free, it follows that $|\mathcal{K}_r(G)| \geq |\mathcal{K}_r(K_{k-1}\vee T_2(n-k+1))|$.
  
  Let $s$ be the minimum number such that there exists a vertex subset $U \subseteq V$ of size $s$ such that $G-U$ is $F$-free. 
  Observe that $s \leq (k-1)|F|$ as if $K$ is a copy of $(k-1) F$, then $G[V \backslash V(K)]$ must be $F$-free since $G$ is $k F$-free.
  Let $G'=G-U$. 
  Note that $F$ does not contain $\Theta(1,2,2,3)$ as a subgraph, we have $|\mathcal{K}_r(G')|=o(n^2)$ by \cref{thm:jun2} and \cref{lem:k-tree}(2).
  Let $T$ be the graph obtained by deleting a root vertex from $F$. Note that $T$ is a tree.
  Define \[U_1=\{u\in U:~((k-1)|F|+1)T\subseteq G'[N(u)\cap V(G')]\}\quad \text{and}\quad U_2=U\backslash U_1.\].

\begin{claim}\label{cl:11111}
  For any $u\in U_2$, the number of $r$-cliques containing $u$ is $O(n)$.
\end{claim}

\begin{proof}
  From \cref{cor:tree}, for each $1\leq i\leq r$, the number of $i$-cliques in $G'[N(u)\cap V(G')]$ is $O(n)$, since $G'[N(u)\cap V(G')]$ is $((k-1)|F|+1)T$-free.
  Recall that $|U|\leq (k-1)|F|=O(1)$.
  Hence, for each $i$, the number of $r$-cliques containing $u$, $i$ vertices in $G'$ and $r-1-i$ vertices in $U-\empty{u}$ is $O(n)$.
  That is, the number of $r$-cliques containing $u$ is $O(n)$.
\end{proof}

\begin{claim}\label{cl:cl22222}
  $|U_1|=k-1$.
\end{claim}

\begin{proof}
   If $|U_1|\geq k$, we can recursively find $k$ disjoint copies of $F$, where each $F$ has a root vertex in $U_1$ and the remaining vertices in $G'$, a contradiction. 
   Indeed, assume we have found $i\le k-1$ disjoint copies of $F$. Then we have selected at most $(k-1)|F|$ vertices in $G'$. 
   Pick $u\in U_1$ which have not selected. 
   By the definition of $U_1$, $N(u)\cap G'$ contains $(k-1)|F|+1$ disjoint copies of $T$. 
   Hence, there is at least one $T$ such that its vertices do not intersect with the selected $iF$
  , so $u$ and this copy of $T$ form an $F$.

   If $|U_1|\leq k-2$, we will show that $|\mathcal{K}_r(G)|<|\mathcal{K}_r(K_{k-1}\vee T_2(n-k+1))|$, a contradiction.
   Recall that $|\mathcal{K}_r(G')|=o(n^2)$.
   Moreover, since $G'$ is $F$-free, we have $e(G')\leq (n-s)^2/4$ by \cref{thm:simonovits}.  
   Hence, the $r$-cliques $R$ in graph $G-U_2$ can be divided into three cases:
   
\begin{itemize}

\item $|R\cap V(G')|\le 1$. The number of this kind of $r$-clique is $O(n)$.

\item $|R\cap V(G')|=2$. The number of this kind of $r$-clique is at most $\binom{k-2}{r-2}(n-s)^2/4$.

\item $|R\cap V(G')|\ge 3$. The number of this kind of $r$-clique is $o(n^2)$.

\end{itemize}

 By \cref{cl:11111}, the number of $r$-cliques that contain at least one vertex in $U_2$ is $O(n)$. 
 Hence,
\[ |\mathcal{K}_r(G)|\leq \binom{k-2}{r-2}(n-s)^2/4+o(n^2)<\binom{k-1}{r-2}(n-k+1)^2/4\le |\mathcal{K}_r(K_{k-1}\vee T_2(n-k+1))|\]
  when $n$ is sufficiently large. Here, $\binom{k-2}{r-2}<\binom{k-1}{r-2}$ because $r\le k+1$.
\end{proof}

\begin{claim}\label{cl:emptyset}
  $U_2=\varnothing$.
\end{claim}

\begin{proof}
  Suppose that $v\in U_2$.
  By the minimality of $U$, there is a copy of $F$ in $G-(U\backslash\{v\})$ that contains $v$.
  Since there are $k-1$ vertices in $U_1$ from \cref{cl:cl22222}, we can recursively find other $k-1$ disjoint copies of $F$ to form $kF$, a contradiction.
\end{proof}

  Combining \cref{cl:cl22222} and \cref{cl:emptyset}, we have $|U|=|U_1|=k-1$.
  Thus,
\begin{align*}\label{formula2}
  |\mathcal{K}_r(G)|
  &\leq \sum_{i=0}^{r} \binom{k-1}{i}|\mathcal{K}_{r-i}(G')|\notag\\
  &\leq \binom{k-1}{r} + \binom{k-1}{r-1}(n-k+1) + \binom{k-1}{r-2}e(G') + \sum_{i=0}^{r-3} \binom{k-1}{i}|\mathcal{K}_{r-i}(G')|\notag\\
  &\leq \binom{k-1}{r} + \binom{k-1}{r-1}(n-k+1) + \binom{k-1}{r-2}\left(\frac{1}{r-2}\sum_{i=0}^{r-3}(e(G') +  c_i|\mathcal{K}_{r-i}(G')|)\right)  \\
  &\leq \binom{k-1}{r} + \binom{k-1}{r-1}(n-k+1) + \binom{k-1}{r-2}\lfloor (n-k+1)^2/4\rfloor \quad \text{(by \cref{lmm:theta})}\notag\\
  &=|\mathcal{K}_r(K_{k-1}\vee T_2(n-k+1))|\notag
\end{align*}
  and equality holds if and only if $G=K_{k-1}\vee T_2(n-k+1)$, where $c_i=(r-2)\binom{k-1}{i}/\binom{k-1}{r-2}$.

\medskip
\noindent{\textbf{Case 2.}} $r\ge k+2.$

\vspace{0.2cm} 

Let $U_1$, $U_2$ and $G'=G-U_1-U_2$ be the same as in Case 1.
  By a similar analysis as in \cref{cl:11111} and \cref{cl:cl22222}, 
  we have 
\begin{itemize}
    \item for any $u\in U_2$, the number of $r$-cliques 
    containing $u$ is $O(n)$, and
    \item $|U_1|\le k-1$.
\end{itemize}
  Hence, there are $O(n)$ $r$-cliques containing at least one vertex from $U_2$.
  
  We now estimate the number of $r$-cliques in $G-U_2$.
  Note that $F$ does not contain $\Theta(1,2,2,3)$ as a subgraph.
  By \cref{thm:jun2} and \cref{lem:k-tree} (2), we have $\ex(n,K_i,F)=o(n^2)$ for any $i\ge 3$.
  Since $r\ge k+2$ and $|U_1|\le k-1$, each $r$-clique in $G-U_2$ has at least three vertices in $G'$.
  That is, the number of $r$-cliques in $G-U_2$ is at most 
  $$\sum\limits_{i=0}^{|U_1|}\binom{|U_1|}{i}|\mathcal{K}_{r-i}(G')|\le 2^{k-1}\sum\limits_{i=3}^{r}\ex(n,K_i,F)=o(n^2).$$
  Hence, $\ex(n,K_r,kF)=|\mathcal{K}_r(G)|=o(n^2)$.
\end{proof}

\section{Concluding remarks}

 The \textit{tree-width} of $G$, denoted by $tw(G)$, is one less than the size of the largest clique in the chordal graph containing $G$ with the smallest clique number. 
The maximal graphs with tree-width exactly $k$ are called $k$-trees, and a graph with tree-width at most $k$ is a subgraph of a $k$-tree (see \cref{ktree}).
  
 Regarding the function $\ex(n, F)$, it is known that $\ex(n, F)=\Theta(n^2)$ when $\chi(F) \geq 3$. 
 In particular, when $\chi(F)=2$, $\ex(n, F)=O(n)$ if and only if $F$ is a forest. 
 Alon and Shikheman \cite{alon2016many} proved that $\ex(n, K_{r},F)=\Theta(n^r)$ if and only if $\chi(F) \geq r+1$, and \cref{lem:k-tree} states that $\ex(n, K_{r},F)=O(n^{r-1})$ when $tw(F)\le r-1$. Perhaps the following is true.
\begin{problem}
  Is it true that if $\chi(F)=r\ge 3$ and $tw(F)\ge r$, then $\ex(n,K_r,F)=\Omega(n^{r-1})$?
\end{problem}

\paragraph{Acknowledgements}

The authors want to thank Suyun Jiang for helpful discussion. The authors are also grateful to Hong Liu and Oleg Pikhurko for careful reading and some writing suggestions.
The authors would also like to thank the anonymous referees for their careful reading and constructive comments.

\bibliographystyle{abbrv}
\bibliography{references.bib}
\end{document}